\newcommand{\powser}[1]{[\![#1]\!]}
\newcommand{\pdiv}{$p$-divisible }

\newcommand{\G}{\mathbb{G}}

\newcommand{\Q}{\mathbb{Q}}
\newcommand{\C}{\mathbb{C}}
\newcommand{\Z}{\mathbb{Z}}
\newcommand{\QZ}{\Q_p/\Z_p}

\newcommand{\al}{\alpha}
\newcommand{\Lk}{\Lambda_k}

\newcommand{\lra}[1]{\overset{#1}{\longrightarrow}}

\newcommand{\Prod}[1]{\underset{#1}{\prod}}

\newcommand{\Coprod}[1]{\underset{#1}{\coprod}}
\newcommand{\Colim}[1]{\underset{#1}{\colim}}

\newcommand{\E}{E_{n}}
\newcommand{\LE}{L_{K(t)}E_n}

\newcommand{\Loops}{\mathcal{L}}
\def \mmod{/\mkern-3mu /}

\date{\today}
\documentclass[10pt]{amsart}
\usepackage{amssymb,amsfonts,amsthm,amsmath,verbatim,tikz-cd,lscape}
\usepackage[margin=1.5in]{geometry}
\usepackage[enableskew]{youngtab}
\usepackage[all]{xy}
\usepackage{ytableau}
\theoremstyle{definition}

\DeclareMathOperator{\pr}{pr}
\DeclareMathOperator{\im}{im}

\DeclareMathOperator{\Hom}{Hom}
\DeclareMathOperator{\colim}{colim}

\DeclareMathOperator{\Iso}{Iso}
\DeclareMathOperator{\Spec}{Spec}
\DeclareMathOperator{\Spf}{Spf}

\DeclareMathOperator{\Sub}{Sub}
\DeclareMathOperator{\Div}{Div}
\DeclareMathOperator{\Level}{Level}

\DeclareMathOperator{\Epi}{Epi}
\DeclareMathOperator{\Tor}{Tor}
\DeclareMathOperator{\Ext}{Ext}

\newtheorem{thm}[subsection]{Theorem}
\newtheorem{prop}[subsection]{Proposition}
\newtheorem{cor}[subsection]{Corollary}
\newtheorem{lemma}[subsection]{Lemma}

\newtheorem*{mainthm}{Theorem}
\newtheorem*{maincor}{Corollary}
\newtheorem{definition}[subsection]{Definition}

\newtheorem{example}[subsection]{Example}

\newtheorem*{remark}{Remark}

\begin{document}
\title{A transchromatic proof of Strickland's theorem}
\author{Tomer M. Schlank}
\author{Nathaniel Stapleton}

\begin{abstract}
In \cite{etheorysym} Strickland proved that the Morava $E$-theory of the symmetric group has an algebro-geometric interpretation after taking the quotient by a certain transfer ideal. This result has influenced most of the work on power operations in Morava $E$-theory and provides an important calculational tool. In this paper we give a new proof of this result as well as a generalization by using transchromatic character theory. The character maps are used to reduce Strickland's result to representation theory.
\end{abstract}

\maketitle
\section{Introduction and outline}
The coefficient ring of Morava $E$-theory carries the universal deformation of a height $n$ formal group over a perfect field of characteristic $p$. This formal group seems to determine the Morava $E$-theory of a large class of spaces. An example of this is the important result of Strickland's \cite{etheorysym} that describes the $E$-theory of the symmetric group (modulo a transfer ideal) as the scheme that classifies subgroups in the universal deformation. This result plays a critical role in the study of power operations for Morava $E$-theory \cite{rezkpowerops, rezkkoszul, rezkcongruence} and explicit calculations of the $E$-theory of symmetric groups \cite{rezkpowercalc, zhupowercalc} and the spaces $L(k)$ \cite{BehrensRezk}.

In this series of papers we exploit a method that reduces facts such as the existence of Strickland's isomorphism into questions in representation theory by using the transchromatic generalized character maps of \cite{tgcm}. In this paper we illustrate the method by giving a new proof of Strickland's result as well as a generalization to wreath products of abelian groups with symmetric groups. We stress that the new feature here is more than the generalization of Strickland's result to certain \pdiv groups, it is a method for reducing a class of hard problems in $E$-theory to representation theory.


We explain some of the ideas. Let $G$ be a finite group. There is an endofunctor of finite $G$-CW-complexes $\Loops$ called the ($p$-adic) inertia groupoid functor that has some very useful properties:
\begin{itemize}
\item Given a cohomology theory $E_G$ on finite $G$-CW-complexes, the composite $E_G(\Loops(-))$ is a cohomology theory on finite $G$-CW-complexes.
\item Let $\ast$ be a point with a $G$-action. There is an equivalence
\[
EG \times_G \Loops(\ast) \simeq \Hom(B\Z_p, BG).
\]
The right hand side is the ($p$-adic) free loop space of $BG$.
\item If $E$ is $p$-complete, characteristic $0$, and complex oriented with formal group $\G_{E}$ then (working Borel equivariantly) there are isomorphisms
\[
E_{\Z/p^k}^0(\Loops(\ast)) \cong E^0(\Coprod{\Z/p^k}B\Z/p^k) \cong \Prod{\Z/p^k}E^0(B\Z/p^k).
\]
This implies that, as $k$ varies, the algebro-geometric object associated to $E(\Loops(-))$ is the \pdiv group $\G_{E}\oplus \QZ$.
\item The target of the character maps of \cite{hkr} and \cite{tgcm} take values in a cohomology theory built using $\Loops$.
\end{itemize}
Because of the second property we feel safe abusing notation and writing $EG \times_G \Loops(\ast)$ and $\Loops BG$ interchangeably. The latter is certainly easier on the eyes.

Now let $E$ be Morava $\E$. The \pdiv group associated to $\G_{E}$ is the directed system built out of the $p^k$-torsion
\[
\G_{E}[p] \rightarrow \G_{E}[p^2] \rightarrow \ldots.
\]
Precomposing with the inertia groupoid $h$ times provides a cohomology theory $E(\Loops^h(-))$ with associated \pdiv group $\G_{E} \oplus \QZ^h$, where $\QZ^h = (\QZ)^h$. In \cite{hkr} and \cite{tgcm} rings called $C_t$ for $0 \leq t < n$ are constructed with three important properties:
\begin{itemize}
\item The ring $C_t$ is a faithfully flat $L_{K(t)}E^0$-algebra and further $E^0$ injects into $C_t$.
\item There is an isomorphism of \pdiv groups
\[
C_t \otimes_{E^0} (\G_{E} \oplus \QZ^h) \cong (C_t \otimes_{L_{K(t)}E^0} \G_{L_{K(t)}E})\oplus \QZ^{n+h-t}.
\]
\item Furthermore, for $X$ a finite $G$-CW-complex there is an isomorphism
\[
C_t \otimes_{E^0} E_{G}^0(X) \cong C_{t} \otimes_{L_{K(t)}E^0} (L_{K(t)}E)_{G}^0(\Loops^{n-t}X).
\]
\end{itemize}

In this paper we are interested in comparing two schemes. The first is corepresented by the cohomology of the symmetric group $E^0(\Loops^hB\Sigma_{p^k})$ modulo the ideal $I_{tr}$ generated by the image of the transfer maps along the inclusions $\Sigma_i \times \Sigma_j \subset \Sigma_{p^k}$, where $i,j>0$ and $i+j = p^k$. Thus the scheme is the functor
\[
\Spec(E^0(\Loops^hB\Sigma_{p^k})/I_{tr}):E^0\text{-algebras} \lra{} \text{Set}
\]
mapping
\[
R \mapsto \Hom_{E^0\text{-alg}}(E^0(\Loops^hB\Sigma_{p^k})/I_{tr},R).
\]
The second scheme classifies subgroups of the \pdiv group $\G_{E} \oplus \QZ^h$ of order $p^k$. It is the functor
\[
\Sub_{p^k}(\G_{E}\oplus \QZ^h):E^0\text{-algebras} \lra{} \text{Set}
\]
mapping
\[
R \mapsto \{H \subseteq R \otimes \G_{E}\oplus\QZ^h||H|=p^k\}.
\]
We begin by proving that these schemes are both finite and flat over $E^0$. We construct a map between the schemes by using properties of the ring $C_0$ to embed the rings of functions on these schemes as lattices inside the  (generalized) class functions on $\Sigma_{p^k}$. By embedding both rings in the same large ring we are able to see that one is a sublattice of the other. Because $C_0$ is faithfully flat over $p^{-1}E = L_{K(0)}E$ the map we construct is an isomorphism after inverting $p$.

We prove that the map is an isomorphism by using the third property of the ring $C_1$ to reduce the computation to height $1$. It suffices to prove the map is an isomorphism after reduction to height $1$ because the schemes are finite and flat and the determinant of the map between them is a power of $p$ times a unit. We wish to show that the power of $p$ is zero. Then we extend our coefficients further to identify with a faithfully flat extension of $p$-adic $K$-theory. Thus we have produced a character map from $E$ to a form of $p$-adic $K$-theory. This allows us to reduce to a problem in representation theory that ends up being equivalent to the canonical isomorphism
\[
\Spec(RA) \lra{\cong} \Hom(A^*,\G_m),
\]
where $RA$ is the representation ring of a finite abelian group $A$, $\G_m$ is the multiplicative group scheme, and $A^*$ is the Pontryagin dual.

The main result of the paper is the following:
\begin{mainthm}
Let $\G_{\E}$ be the universal deformation formal group, let $\Loops$ be the ($p$-adic) free loop space functor, and let
\[
\Sub_{p^k}(\G_{\E}\oplus \QZ^h):\E^0\text{-algebras} \lra{} \text{Set}
\]
be the functor mapping
\[
R \mapsto \{H \subseteq R \otimes \G_{\E}\oplus\QZ^h||H|=p^k\}.
\]
There is an isomorphism
\[
\Spec(\E^0(\Loops^hB\Sigma_{p^k})/I_{tr}) \cong \Sub_{p^k}(\G_{\E}\oplus \QZ^h),
\]
where $I_{tr}$ is the ideal generated by the image of the transfer along the inclusions $\Sigma_i \times \Sigma_j \subseteq \Sigma_{p^k}$.
\end{mainthm}

Let $\al:\Z_{p}^{h} \lra{} \Sigma_{p^k}$ and let
\[
\text{pr}:\Sub_{p^k}(\G_{\E}\oplus \QZ^h) \lra{} \Sub_{\leq p^k}(\QZ^h)
\]
be the projection onto the constant \'etale factor. For $A \subset \QZ^h$ let
\[
\Sub_{p^k}^{A}(\G_{\E}\oplus \QZ^h):\E^0\text{-algebras} \lra{} \text{Set}
\]
be the functor mapping
\[
R \mapsto \{H \subseteq R \otimes \G_{\E}\oplus\QZ^h||H|=p^k, \text{pr}(H) = A\}.
\]
We say a map $\al:\Z_{p}^{h} \lra{} \Sigma_{p^k}$ is monotypical if the corresponding $\Z_{p}^h$-set of size $p^k$ is a disjoint union of isomorphic  transitive $\Z_{p}^h$-sets. Assume that $\al$ is monotypical and let $A = \im \al$, then $C(\im \al) \cong A \wr \Sigma_{p^j}$. Let $I_{tr} \subset \E^0(BA\wr\Sigma_{p^j})$ be the ideal generated by the image of the transfers along $A \wr (\Sigma_l \times \Sigma_m) \subset A \wr \Sigma_{p^j}$ with $l,m>0$ and $l+m = p^j$.

A corollary of the theorem above is the $E$-theory version of the second authors Theorem 3.11 from \cite{subpdiv}.
\begin{maincor}
Let $\al:\Z_{p}^{h} \lra{} \Sigma_{p^k}$ be monotypical, $A$ the image of $\al$, and $p^j = p^k/|A|$. Then there are isomorphisms
\[
\Spec(\E^0(BA\wr\Sigma_{p^j})/I_{tr}) \cong \Spec(\E^0(BC(\im \al)/I_{tr}^{[\al]}) \cong \Sub_{p^k}^{A^*}(\G_{\E}\oplus \QZ^h),
\]
where $I_{tr}^{[\al]}$ is described in Theorem \ref{transferthm} below.
\end{maincor}

\paragraph{\emph{Outline}} In Section \ref{s:characters} we review the transchromatic character maps and we show that, for good groups, the character map from height $n$ to height $1$ can be modified to land in a faithfully flat extension of $p$-adic $K$-theory.

Then in Section \ref{unitary} we develop a transchromatic character theory for the $E$-theory of the unitary group. We provide an algebro-geometric interpretation of the character map in terms of divisors of $\G_{\E}$ and divisors of $C_{t}\otimes \G_{\LE} \oplus (\Z/p^k)^h$.

In Section \ref{s:Groups} we develop the theory of centralizers in symmetric groups. These arise in the decomposition of the iterated free loop space $\Loops^hB\Sigma_{p^k}$. The centralizers that we are interested in all have the form $A\wr \Sigma_{p^j}$ with $j<k$. In order to understand certain transfer maps, we also study the free loops of the maps $B\Sigma_i \times B\Sigma_j \rightarrow B\Sigma_{p^k}$, where $i,j>0$ and $i+j=p^k$.

In Section \ref{s:ETheoryisFree} we show that $\E^0(BA\wr \Sigma_{p^j})/I_{tr}$ is finitely generated and free as an $\E^0$-module. This relies on work of Rezk in \cite{rezkkoszul} that reduces the question to Strickland's result (Thm 8.6, \cite{etheorysym}) that $\E^0(B\Sigma_{p^k})/I_{tr}$ is finite and free. 

In Section \ref{s:SubGrpisFree} we show that the ring of functions on $\Sub_{p^k}^{A}(\G_{\E} \oplus \QZ^h)$ is free of the same rank as an $\E^0$-module. This follows by reduction to Strickland's result (Thm 10.1, \cite{subgroups}) concerning $\Sub_{p^k}(\G_{\E})$. Thus in Sections \ref{s:ETheoryisFree} and \ref{s:SubGrpisFree} we rely on two freeness results of Strickland's.

Finally in Section \ref{s:GenStrickland} we construct an injective map
\[
f_{p^k}:\Gamma \Sub_{p^k}(\G_{\E}\oplus \QZ^h) \hookrightarrow \E^0(\Loops^hB\Sigma_{p^k})/I_{tr}
\]
by embedding both of the rings into the ring of class functions on $(n+h)$-tuples of commuting elements in $\Sigma_{p^k}$ and exhibiting the domain as a subset of the codomain. The map has the property that it becomes an isomorphism after inverting $p$. Since the domain and codomain are free of the same rank this means that the failure of the determinant to be a unit is only a power of the prime $p$. Thus, to prove that the map is an isomorphism, it suffices to base change the map to any $\E^0$-algebra in which $p$ is not a unit and is not nilpotent and prove that the resulting map is an isomorphism. We then use the transchromatic character maps to reduce to height $1$ and further base change in order to identify with a form of $p$-adic $K$-theory. This converts proving the map is an isomorphism to a question in representation theory that is easy to solve.

There is an appendix that includes an elementary proof of Strickland's theorem when $k = 1$.

\paragraph{\emph{Acknowledgements}} It is a pleasure to thank Jacob Lurie for suggesting reduction to height $1$. With a sentence he initiated this project. We would like to thank Tobias Barthel for many useful conversations and for his interest from the very start. We have been strongly influenced by the work of Charles Rezk and Neil Strickland and we are grateful to Charles Rezk for several discussions. We also thank Mark Behrens, Pavel Etingof, Tyler Lawson, Haynes Miller, Peter Nelson, Eric Peterson, and Vesna Stojanoska for their many helpful remarks. The first author was supported by the Simons foundation and the second author was partially supported by NSF grant DMS-0943787.
\section{From $E$-theory to $K$-theory} \label{s:characters}
We recall the basics of the transchromatic character maps. Then we modify the existing character map that lands in height $1$ to land in a faithfully flat extension of $p$-adic $K$-theory.
\subsection{Character theory recollections} \label{ss:recollections}

Fix a prime $p$. We fix inclusions
\[
\QZ \subset \Q/\Z \subset S^1 \subset \C^*,
\]
where the middle inclusion is the exponential map $e^{2\pi i x}$. Thus for a finite abelian group $A$ there is a canonical isomorphism between the characters of $A$ and the Pontryagin dual of $A$
\[
\hat{A} \cong A^*.
\]
Throughout the rest of the paper we will use the notation $A^*$ for either of these.
 
Now let $G$ be a finite group and $X$ a finite $G$-CW complex. We may produce from this the topological groupoid $X \mmod G$. Let $|X\mmod G|$ be the geometric realization of the nerve of $X \mmod G$. Thus we have an equivalence
\[
|X \mmod G| \simeq EG \times_G X.
\]
For any cohomology theory $E$, in this paper, we always set
\[
E^*(X\mmod G) = E^*(|X\mmod G|) = E^*(EG \times_G X).
\]
There is an endofunctor of topological groupoids called the ($p$-adic) inertia groupoid functor that we will denote by $\Loops$. It takes finite $G$-CW complexes to finite $G$-CW complexes by mapping
\[
X \mmod G \mapsto \Loops(X \mmod G) := \Hom_{top.gpd}(* \mmod \Z_{p},X \mmod G),
\]
where the right hand side is the internal mapping topological groupoid. It is a $G$-space by the isomorphism
\[
\Hom_{top.gpd}(* \mmod \Z_{p},X \mmod G) \cong \big( \Coprod{\al \in \Hom(\Z_p,G)}X^{\im \al} \big)\mmod G,
\]
where $X^{\im \al}$ is the fixed points of $X$ with respect to the image of $\al$ and $G$ acts by mapping $x \in X^{\im \al}$ to $gx \in X^{\im g \al g^{-1}}$.

It is notationally convenient to equate $\Loops(X \mmod G)$ and $\Loops(EG \times_G X)$ and we will use these interchangeably. Note that $\Loops(EG \times_G X)$ is not quite (but is closely related to) the free loop space of $EG \times_G X$. In particular, when $X$ is a space with an action by the trivial group then $\Loops X = X$ and when $X = \ast$ and $G$ is a $p$-group $\Loops(X \mmod G) = \Hom(S^1, BG)$.

For a finite group $G$, the transchromatic generalized character maps approximate the Morava $E$-theory of $BG$ by a certain height $t$ cohomology theory for any $t < n$. For arbitrary $t$ these are constructed in \cite{tgcm}, when $t = 0$ this is in \cite{hkr}, and when $n=1$ it can be found in \cite{Adams-Maps2}. These papers construct a faithfully flat extension of $\LE$ called $C_{t}^{0}$ (for $t>0$ the faithfully flatness is Proposition \ref{p:ff} below) and a map of cohomology theories on finite $G$-CW complexes called ``the character map"
\[
\E^0(X\mmod G) \lra{} C_{t}^{0} \otimes_{\LE^0}\LE^0(\Loops^{n-t}X \mmod G).
\]

The main theorem regarding these maps is the following:
\begin{thm} \label{charactermap} \cite{tgcm}
The character map has the property that the map induced by tensoring the domain up to $C_t$
\[
C_{t}^0 \otimes_{\E^0}\E^0(X\mmod G) \lra{\cong} C_{t}^0 \otimes_{\LE^0}\LE^0(\Loops^{n-t}X \mmod G)
\]
is an isomorphism.
\end{thm}


There is a decomposition
\[
\Loops^h BG \cong \Coprod{[\al] \in \Hom(\Z_{p}^h, G)/\sim} BC(\im \al),
\]
where $[\al]$ is the conjugacy class of a map $\al:\Z_{p}^h \lra{} G$ and $C(\im \al)$ is the centralizer of the image of $\al$.

\begin{cor}
When $X = *$ the character map takes the form
\[
\E^0(BG) \lra{} \Prod{[\al] \in \Hom(\Z_{p}^{n-t}, G)/\sim} C_{t}^0 \otimes_{\LE^0} \LE^0(BC(\im \al)).
\]
\end{cor}

\begin{remark}
When $t=0$ the ring $C_{0}^0$ is a $L_{K(0)}\E^0 = p^{-1}\E^0$-algebra and so $C_{0}^{0}(BG) = C_{0}^0$ for all finite $G$. Thus the codomain of the character map is just the product
\[
\Prod{[\al] \in \Hom(\Z_{p}^n, G)/\sim}C_{0}^0,
\]
which is the ring of generalized class functions on $G$ with values in $C_{0}^{0}$. Because of this we will often rewrite this ring as $Cl(G_{p}^{n},C_{0}^{0})$, where $G_{p}^{n}$ is shorthand for $\Hom(\Z_{p}^{n},G)$.
\end{remark}

Furthermore, there is a version of the theorem that includes taking the quotient by a transfer along $i:H \subset G$. Let $I_{tr} \subset \E^0(BG)$ be the image of the transfer along $i$. Let $\al : \Z_p \rightarrow G$ and let $\al '$ be a lift of $\al$ to $H$ up to conjugacy so that there exists $g \in G$ so that $gi\al ' g^{-1} = \al$. There is an induced inclusion
\[
gC_H(\im \al ')g^{-1} \lra{} C_G(\im \al).
\]
Define the ideal $I_{tr}^{[\al]} \subseteq C_{t}^{0}\otimes_{\LE^0} \LE^0(BC_G(\im \al))$ to be the ideal generated by the image of the transfer along the inclusions induced by the lifts of $\al$ to $H$ up to conjugacy.

\begin{thm} \label{transferthm} (Theorem 2.18, \cite{subpdiv})
The character map induces an isomorphism
\[
C_{t}^0 \otimes_{\E^0} \E^0(BG)/I_{tr} \cong \Prod{[\al] \in G_{p}^{n-t}/\sim}C_{t}^{0}\otimes_{\LE^0} \LE^0(BC(\im \al))/I_{tr}^{[\al]}.
\]
\end{thm}

\begin{remark}
There is an analogous result for the cohomology theory $\E^0(\Loops^h(-))$:
\[
\E^0(\Loops^h BG)/I_{tr} \cong \Prod{[\al] \in G_{p}^{h}/\sim}\E^{0}(BC(\im \al))/I_{tr}^{[\al]}.
\]
\end{remark}

\subsection{Landing in $K$-theory} \label{sub:K}
In this subsection we prove that the ring $C_{t}^0$ is faithfully flat as an $\LE^0$-algebra. With this result we show that for good groups $G$, extension to $\pi_0(L_{K(1)}(C_1 \wedge K_p))$ identifies the codomain of the character map with a faithfully flat extension of $K_p$. 


Let $\G := \LE^0 \otimes \G_{\E}$ so that $\G_{\LE}$ is the connected component of the identity of $\G$ (Proposition 2.4 of \cite{tgcm}). Recall the following proposition.
\begin{prop}\label{num2} (Proposition 2.17 in \cite{tgcm})
The functor from $\LE^0$-algebras to sets given by
\[
\Iso_{\G_{\LE} / }(\G_{\LE}\oplus\QZ^{n-t},\G) : R \mapsto \Iso_{\G_{\LE} / }(R \otimes \G_{\LE}\oplus\QZ^{n-t},R \otimes \G)
\]
is representable by $C_{t}^0$.
\end{prop}

The following result should be compared to Proposition 6.5 of \cite{hkr}.
\begin{prop} \label{p:ff}
The ring $C_{t}^{0}$ is a faithfully flat $\LE^0$-algebra.
\end{prop}
\begin{proof}
Note that $C_{t}^0$ is a flat $\LE^0$-algebra since it is constructed as a colimit of algebras each of which is a localization of a finitely generated free module. 

We now show that it is surjective on $\Spec(-)$. Let $P \subset \LE^0$ be a prime ideal. Let $i$ be the smallest natural number such that $u_i \notin P$. Note that $i \leq t$. Now consider 
\[
K := \overline{(\LE^0/P)}_{(0)}
\]
the algebraic closure of the fraction field. There is an isomorphism (pg. 34, \cite{Dem})
\[
K \otimes \G \cong \G_{\text{for}} \oplus \QZ^{n-i}.
\]
The formal part has height $i$ because $u_i$ has been inverted. Now since
\[
\G_{\text{for}} \oplus \QZ^{n-i} \cong \G_{\text{for}} \oplus \QZ^{t-i} \oplus \QZ^{n-t},
\]
Proposition \ref{num2} implies that this is classified by a map $C_{t}^0 \lra{q} K$ that extends the canonical map $\LE^0 \lra{} K$. Now the kernel of $q$ must restrict to $P$.
\end{proof}

Recall the following definition:
\begin{definition}(Definition 7.1, \cite{hkr})
A finite group is \emph{good} if $K(n)^*(BG)$ is generated by transfers of Euler classes of complex representations of subgroups of $G$.
\end{definition}

If $G$ is good then $K(n)^*(BG)$ is concentrated in even degrees (by the remark after Definition 7.1 in \cite{hkr}) and this implies (Proposition 3.5 in \cite{etheorysym}) that $\E^*(BG)$ is free and concentrated in even degrees.

\begin{cor} \label{c:fgp}
For a good group $G$ the ring $\LE^0(BG)$ is finitely generated and projective as an $\LE^0$-module. 
\end{cor}
\begin{proof}
Since $G$ is good $\E^0(BG)$ is finitely generated and free. Theorem \ref{charactermap} implies that $C_{t}^{0}\otimes_{\LE^0} \LE^0(BG)$ is a summand of 
\[
C_{t}^{0} \otimes_{\E^0} \E^0(BG), 
\]
which is finitely generated and free. Now faithfully flat descent for finitely generated projective modules (10.78.3 in \cite{stacks}) and Proposition \ref{p:ff} implies that $\LE^0(BG)$ is finitely generated and projective.
\end{proof}

The cohomology theory $C_{t}^{0} \otimes_{\LE^0} \LE^*(-)$ defined on finite spaces gives rise to a spectrum $C_t$ and thus a cohomology theory $C_{t}^*(-)$ defined on all spaces. 

\begin{lemma}
The ring spectrum $C_t$ is $E_{\infty}$.
\end{lemma}
\begin{proof}
We may handcraft an $E_{\infty}$-ring spectrum $C_{t}'$ with coefficient ring $C_{t}^*$. Let $\Lk = (\Z/p^k)^{n-t}$. Recall from Subsection 2.8 of \cite{tgcm} that $$C_{t}^{0} = \Colim{k} \text{ } \pi_0 C_{t,k},$$ where $C_{t,k} = S_{k}^{-1} (L_{K(1)}\E \wedge_{\E} \E^{B\Lk})$. Recall that $S_k$ is a finite collection of elements in $\pi_0 (\LE \wedge_{\E} \E^{B\Lk})$ determined by the map $\Lk^* \lra{} \G_{et}[p^k]$. Thus $C_{t}' = \Colim{k} \text{ } C_{1,k}$ has the correct homotopy. Now $C_t$ and $C_{t}'$ are both Landweber exact cohomology theories, thus they are determined up to equivalence by their coefficient rings by Proposition 2.8 in \cite{hoveystrickland}.
\end{proof}


The following result is a special case of a result of Hopkins'.
\begin{prop} \cite{coctalos}
The ring $\pi_0(C_1 \wedge K_p)$ represents the scheme $\Iso(\G_{C_1},\G_m)$.
\end{prop}

Let $\mathcal{M}_{fg}$ be the moduli stack of formal groups. This result follows from the fact that the maps classifying the formal groups $\Spec(C_{1}^0) \lra{} \mathcal{M}_{fg}$ and $\Spec(K_{p}^{0}) \lra{} \mathcal{M}_{fg}$ are both flat. Further this implies that both of the maps $C_{1}^{0} \lra{} \pi_0(C_1 \wedge K_p)$ and $\Z_p = K_{p}^{0} \lra{} \pi_0(C_1 \wedge K_p)$ are flat.

\begin{prop}
The ring $\pi_0(C_1 \wedge K_p)$ is non-zero and $p \notin \pi_0(C_1 \wedge K_p)^{\times}$.
\end{prop}
\begin{proof}
To prove both facts it suffices to construct a non-zero ring in which $p$ is not invertible and over which the formal group laws associated to $C_1$ and $K_p$ are isomorphic. The separable closure of the fraction field of $C_{1}^0/p$ does the trick by A2.2.11 in \cite{greenbook}.
\end{proof}

\begin{cor}
The ring $\pi_0(C_1 \wedge K_p)$ is a faithfully flat $\Z_p$-algebra.
\end{cor}
\begin{proof}
Since $\pi_0(C_1 \wedge K_p)$ is flat over $\Z_p$, no power of $p$ can go to zero in $\pi_0(C_1 \wedge K_p)$. Thus the map from $\Z_p$ is injective. Since $p$ is not invertible $p\pi_0(C_1 \wedge K_p) \neq \pi_0(C_1 \wedge K_p)$, which implies the map is faithfully flat by Theorem 7.2 in \cite{matsumura}.
\end{proof}

Let $\bar{C}_1 = L_{K(1)}(C_1 \wedge K_p)$. Since $C_1 \wedge K_p$ is $L_1$-local, $K(1)$-localization is just $p$-completion. The composite $\Z_p \lra{} \bar{C}_{1}^{0}$ is faithfully flat. This is certainly not  obvious since $\pi_0({C}_{1} \wedge K_p)$ may not be Noetherian. However, it does follow by some recent work on completions (see 1.3 and 1.13 in \cite{hoveysomess} and A.15 in \cite{barthelfrankland}). We wish to thank Tobias Barthel for suggesting we consider dualizability to prove the following result:



\begin{prop} \label{p.basechange}
Let $G$ be a good group and let $\bar C_1 = L_{K(1)}(C_1 \wedge K_p)$. The maps $K_p \rightarrow \bar{C}_1$ and $C_1 \rightarrow \bar{C}_1$ induce isomorphisms
\[
\bar{C}_{1}^{0}\otimes_{\Z_p} K_{p}^{0}(BG) \cong \bar{C}_{1}^{0}\otimes_{C_{1}^{0}} (C_{1}^{0}\otimes_{L_{K(1)}\E^0} L_{K(1)}\E^0(BG)) \cong \bar{C}_{1}^{0}(BG).
\]
\end{prop}
\begin{proof}
For the purposes of this proof let $L_1 := L_{K(1)}\E$. Theorem 4.1 of \cite{ekmm} provides a conditionally convergent spectral sequence
\[
\Ext_{L_{1}^*}^{p,q}(\pi_{-*} L_{K(1)}(L_1 \wedge BG), \bar C_{1}^{*}) \Longrightarrow \pi_{-p-q}(F_{L_1}(L_{K(1)}(L_1 \wedge BG),\bar C_1)).
\]
The target can be rewritten as
\[
F_{L_1}(L_{K(1)}(L_1 \wedge BG),\bar C_1) \simeq F_{L_1}(L_1 \wedge BG,\bar C_1) \simeq F(BG,\bar C_1)
\]
by using the fact $\bar C_1$ is $K(1)$-local. The vanishing of the Tate construction in the $K(1)$-local category \cite{greenlees-sadofsky} and Corollary \ref{c:fgp} imply that
\[
\pi_*L_{K(1)}(L_1 \wedge BG) \cong \pi_*(L_{1}^{BG}) 
\]
is finitely generated and projective. Thus the spectral sequence immediately collapses and we have an isomorphism
\[
\bar C_{1}^{0}(BG) \cong \Hom_{L_{1}^0}(\pi_{0} L_{K(1)}(L_1 \wedge BG),\bar C_{1}^0).
\]
Now finitely generated projective modules are (strongly) dualizable (Example 1.4 \cite{Dold-Puppe}, Section 2 \cite{Axiomatic}). Thus there is an isomorphism
\[
\Hom_{L_{1}^0}(\pi_{0} L_{K(1)}(L_1 \wedge BG),C_{1}^0) \cong \Hom_{L_{1}^0}(\pi_{0} L_{K(1)}(L_1 \wedge BG),L_{1}^0) \otimes_{L_{1}^0} \bar C_{1}^0
\]
and this is isomorphic to 
\[
L_{1}^0(BG) \otimes_{L_{1}^0} \bar C_{1}^0
\]
by the same spectral sequence with $L_1$ in place of $\bar C_1$. The claim follows and the same proof applies to $K_p$ and $\bar C_1$.
\end{proof}


\paragraph{\underline{\emph{Convention}}} Now that we have discussed these points and proved the previous proposition we will return to the second authors conventions in \cite{tgcm} and \cite{subpdiv} and abuse notation and write 
\[
C_{t}^0(X) := C_{t}^0 \otimes_{\LE^0} \LE^0(X)
\]
for all spaces $X$ (even $X=BU(n)$).

\section{Character theory of the unitary group} \label{unitary}

\subsection{The inertia groupoid and the unitary group}

Recall from Subsection \ref{ss:recollections} that the ($p$-adic) inertia groupoid $\Loops$ is defined to be the functor on topological groupoids corepresented by $*\mmod \Z_p$. Since we have taken $G$ to be finite, we can replace $*\mmod \Z_p$ by $*\mmod \Z/p^k$ for $k$ large enough. 

In this subsection we study the inertia groupoid applied to $U(m)$. In order to get a proper grip on $\Loops BU(m)$, we use a ``torsion" version of the inertia groupoid. Thus we define
\[
\Loops^{h}_{k}(X \mmod G) := \Hom_{top.gpd}(* \mmod (\Z/p^{k})^h,X \mmod G),
\]
and we may abuse this notation as described in the previous subsection.

Let $P(m,k)$ be the set of unordered $m$-tuples of $p^k$th roots of unity. For $s \in P(m,k)$ let $s_{\zeta_i}$ be the number of times $\zeta_i$ appears in the tuple. It immediately follows that, for $s \in P(m,k)$,
\[
\sum_{i = 1}^{\infty} s_{\zeta_i} = \sum_{i = 1}^{p^k} s_{\zeta_i} = m.
\]

\begin{prop}
There is an equivalence of spaces
\[
\Loops_{k}(* \mmod U(m)) = \Hom(*\mmod \Z/p^k, * \mmod U(m)) \simeq \Coprod{s \in P(m,k)} * \mmod U(s_{\zeta_1}) \times \ldots \times U(s_{\zeta_{p^k}}). 
\]
\end{prop}
\begin{proof}
The proof of this is elementary. We include a proof for completeness. Every unitary matrix is diagonalizable. If the matrix has order dividing $p^k$ then its eigenvalues are $p^k$th roots of unity. Conjugation by permutation matrices permutes the entries in the diagonal. Thus the conjugacy classes of maps $\Z/p^k \rightarrow U(m)$ are determined by the set of eigenvalues. Now the centralizer of a diagonalized matrix is determined by the number of repeated eigenvalues.
\end{proof}

Let $P(m,k,h)$ be the set of unordered $m$-tuples of ordered $h$-tuples of $p^k$th roots of unity. For $s \in P(m,k,h)$ and $i \in (S^1)^h[p^k]$, let $s_i$ be the number of instances of the $h$-tuple $i$ in $s$. Note that $P(m,k,1) = P(m,k)$.

\begin{prop} \label{loopunitary}
Let $\{i_1, \ldots, i_{p^{h}}\} = (S^1)^h[p^k]$. There is an equivalence of spaces 
\[
\Hom(*\mmod (\Z/p^k)^h, * \mmod U(m)) \simeq \Coprod{s \in P(m,k,h)} * \mmod U(s_{i_1}) \times \ldots \times U(s_{i_{p^{hk}}}). 
\]
\end{prop}
\begin{proof}
This can be seen two ways. It follows from the previous proposition and the adjunction between $\Hom$ and $\times$ for topological groupoids. It also follows from the fact that commuting tuples of unitary matrices can be simultaneously diagonalized.
\end{proof}

\subsection{Character theory and divisors}
Recall from \cite{tgcm} and \cite{subpdiv}, that the \pdiv group associated to $C_{t}^0(\Loops^{n-t}(-))$ is $\G_{C_t}\oplus \QZ^{n-t}$.


Let $E$ be either of the cohomology theories $\E$ or $\LE$.
\begin{prop}
The \pdiv group associated to $E^0(\Loops^{h}(-))$ is $\G_{E} \oplus \QZ^{h}$.
\end{prop}
\begin{proof}
We calculate the effect on $* \mmod \Z/p^k$. We have
\[
E^0(\Loops^h(* \mmod \Z/p^k)) \cong E^0(\Loops^h(B\Z/p^k)) \cong \Prod{(\Z/p^k)^h} E^0(B\Z/p^k) \cong \Gamma ((\G_{E} \oplus \QZ^{h})[p^k]).
\]
It is easy to check that the Hopf algebra structure on these rings is isomorphic as well.
\end{proof}

Now let 
\[
\Lk = ((\Z/p^k)^h)^*,
\]
the Pontryagin dual of the abelian group that we use to define $\Loops_{k}^{h}(-)$. There are many divisors of degree $m$ in $\G_{E} \oplus \QZ^{h}$. To get some control over the set, one can consider divisors of degree $m$ in $\G_{E} \oplus \Lk$.

\begin{definition}
A divisor of degree $m$ in $\G_{E} \oplus \Lk$ is a closed subscheme that is finite and flat of degree $m$ over $E$.
\end{definition}

We will use the isomorphism
\[
\Div_m(\G_{E} \oplus \Lk) \cong (\G_{E} \oplus \Lk)^{\times m}/\Sigma_m.
\]
The quotient by $\Sigma_m$ is the scheme-theoretic quotient. Since $(\G_{E} \oplus \Lk)^{\times m}$ is affine, the quotient is affine as well. When $h=0$, $\Div_m(\G_{E})$ is the usual divisors of degree $m$ in the formal group $\G_{E}$.

\begin{prop}
The functor $\Div_{m}(\G_E \oplus \Lk)$ is corepresented by
\[
E^0(\Loops^{h}_{k}BU(m)).
\]
\end{prop}
\begin{proof}
This is immediate from Proposition \ref{loopunitary} and the fact that, in all of these cases, $E^0(BU(m))$ corepresents divisors of degree $m$ on $\G_{E}$ (Section 9 of \cite{etheorysym}). It is worth noting explicitly what is occuring on the level of connected components. A conjugacy class of maps
\[
(\Z/p^k)^h \lra{} U(m)
\]
determines an $m$-dimensional representation of $(\Z/p^k)^h$ up to isomorphism. The representation decomposes as a sum of $m$ $1$-dimensional representations. This corresponds to $m$ maps $\Z/p^k \lra{} S^1$ which determines $m$ elements in $\Lk$ (counted with multiplicity). The divisor is concentrated on the components of $\G_{E}\oplus\Lk$ corresponding to these elements.
\end{proof}

We may apply the transchromatic character maps to the unitary groups by using $\Loops^{n-t}_{k}$ in place of $\Loops^{n-t}$. 
The character map takes the form
\[
\E^0(BU(m)) \lra{} C_{t}^{0}(\Loops^{n-t}_{k}BU(m)).
\]
These are defined exactly as the character maps for finite groups. There is no difficulty because we are working with the torsion inertia groupoid.

Note that a map $G \lra{f} U(m)$ induces $\Loops^{n-t}_{k}BG \lra{\Loops^{n-t}_{k}f} \Loops^{n-t}_{k}BU(m)$. 

\begin{prop} \label{naturality}
Consider a map $G \lra{f} U(m)$ and let $k$ be large enough so that every map $\Z_p \lra{} G$ factors through $\Z/p^k$. The following square commutes 
\[
\xymatrix{\E^0(\Loops^{h}_{k}BU(m)) \ar[r] \ar[d] & \E^0(\Loops^{h}BG) \ar[d] \\  C_{t}^{0}(\Loops^{h+n-t}_{k}BU(m)) \ar[r] &  C_{t}^{0}(\Loops^{h+n-t}BG).}
\]
\end{prop}
\begin{proof}
The character map is the composite of two maps, one is topological and the other is algebraic. The commutativity of the diagram follows from the fact that the topological part of the character map is $\E^0$ applied to an evaluation map. That is, the following diagram induced by the map $f$ commutes:
\[
\xymatrix{\ast \mmod G \ar[d]^{f} & \ast \mmod \Lk \times \Hom(\ast \mmod \Lk, \ast \mmod G) \ar[d] \ar[l]^-{ev} \\ \ast \mmod U(m) & \ast \mmod \Lk \times \Hom(\ast \mmod \Lk, \ast \mmod U(m)). \ar[l]^-{ev} \\
}
\]
\end{proof}

We often say that character maps like these approximate height $n$ cohomology by height $t$ cohomology because when we tensor the domain up to $C_t$ they give an isomorphism. This is \emph{not true} for $U(m)$. However, these maps do have interesting algebro-geometric content.

There is a canonical map of formal groups $\G_{\LE} \lra{} \G_{\E}$ and over $C_{t}^0$ there is a canonical map $\Lk \lra{} \G_{\E}$ for all $k$. Put together this gives
\[
C_{t}^0 \otimes (\G_{\LE}\oplus \Lk)^{\times m}/\Sigma_m \lra{} (\G_{\E})^{\times m}/\Sigma_m,
\]
which we write as $C_{t}^{0} \otimes \Div_m(\G_{\LE}\oplus \Lk) \lra{} \Div_m(\G_{\E})$. Here we are thinking of $\G_{\LE}$ as a formal group so $C_{t}^0 \otimes \G_{\LE}$ is corepresented by $C_{t}^0(BS^1) = C_{t}^0 \otimes_{\LE^0} \LE^0(BS^1)$.


\begin{prop} \label{unitarytheorem}
Let $\Lk = (\Z/p^k)^{n-t}$. The character map 
\[
\E^0(BU(m)) \lra{} C_{t}^{0}(\Loops^{n-t}_{k}BU(m)).
\]
fits into a commutative square
\[
\xymatrix{\E^0(BU(m)) \ar[r]^{\cong} \ar[d] & \Gamma \Div_{m}(\G_{\E}) \ar[d] \\ C_{t}^{0}(\Loops^{n-t}_{k}BU(m)) \ar[r]^-{\cong} & C_{t}^0 \otimes \Gamma \Div_{m}(\G_{\LE} \oplus \Lk).}
\]
\end{prop}
\begin{proof}
Consider the character map applied to the inclusion of a maximal torus $\mathbb{T} \subset U(m)$. As $\E^0(BU(m))$ injects in $\E^0(B\mathbb{T})$ and $C_{t}^0(\Loops_{k}^{n-t}BU(m))$ injects into $C_{t}^0(\Loops_{k}^{n-t}B\mathbb{T})$, it is enough to check that the character map applied to $S^1$ produces the global sections of 
\[
C_{t}^0 \otimes \G_{\LE} \oplus \Lk \lra{} \G_{\E}
\]
and this is the case.
\end{proof}

\section{Centralizers in symmetric groups} \label{s:Groups}
In this section we develop the theory of centralizers of tuples of commuting elements in wreath products of a finite abelian group and a symmetric group. These arise when studying
\[
\Loops^hB\Sigma_m
\]
and play an important role in the rest of the paper. We show that these centralizers are all products of wreath products of finite abelian groups and symmetric groups. We also analyze $\Loops(-)$ applied to the maps $\Sigma_i \times \Sigma_j \lra{} \Sigma_m$, where $i,j >0$ and $i+j = m$.

For the rest of this section let $A$ be a finite abelian group and let $n\geq 1$ be an integer.
Let $$[n]:= \{1,2,...,n\}$$
and consider the set $${A \times [n]} := \Coprod{1\leq i\leq n} A.$$
The group $A^n$ acts on ${A \times [n]}$ by multiplication coordinate-wise and  the symmetric group $\Sigma_n$ acts on ${A \times [n]}$ by permuting the coordinates.
The two actions fit together to give an action of $A \wr \Sigma_n$ on ${A \times [n]}$.
This action defines a map
$$s:A\wr \Sigma_n \hookrightarrow \Sigma_{|A|n}.$$
The image of the diagonal map

$$z : A \hookrightarrow A\wr \Sigma_n $$
is the center  of $A\wr \Sigma_n$.

Now consider the map
$$d:= s \circ z : A \to \Sigma_{|A|n}.$$
Since the image of $z$ is $Z(A\wr \Sigma_n)$
we have
$$\im s = A \wr \Sigma_n \subset C_{\Sigma_{|A|n}}(\im d). $$
\begin{lemma}\label{l:wr_1}
There is an equality $\im(s) = C_{\Sigma_{|A|n}}(\im d)$ and thus
$$ A \wr \Sigma_n \cong C_{\Sigma_{|A|n}}(\im d). $$
\end{lemma}
\begin{proof}
We can consider the set ${A \times [n]}$ as an $A$-set via the diagonal action. Thus $C_{\Sigma_{|A|n}}(\im d)$ is just the group of automorphisms of ${A \times [n]}$ as an $A$-set.
  As an $A$-set ${A \times [n]}$ is a disjoint union of $n$-copies of $A$ since the group of automorphisms of $A$ as an $A$-set is isomorphic to $A$. The group of automorphisms of $n$-copies is $ A \wr \Sigma_n $.
\end{proof}

Let $h\geq 0$ be an integer and let $$\alpha:\Z^h \to A \wr \Sigma_n$$ be a map. We denote by $\tilde{\alpha}$ the map
$$\tilde{\alpha}:= (s\circ \alpha)\oplus d : \Z^h \oplus A \to \Sigma_{n|A|}. $$
Now consider $\tilde{\alpha}$ as an action of $\Z^h \oplus A$ on ${A \times [n]}$.
Given an element $x\in {A \times [n]}$ and a map $\al$ we define the $\emph{type}$ of $x$ to be the unique surjection
$$t_x: \Z^h \oplus A \twoheadrightarrow A_{t_x}$$ with kernel equal to the stabilizer of $x$. It easy to see that the map $t_x$ induces an inclusion
$$A \hookrightarrow A_{t_x}.$$
Note that if $x$ has type $t$ then so does any other element in the $\tilde{\alpha}$-orbit $x$.
Thus given an $\tilde{\alpha}$ orbit $O$ it makes sense to speak of the type of $O$.
 Given a type $t:\Z^h \oplus A \twoheadrightarrow A_t$  we denote by $N_t$ the number of $\tilde{\alpha}$-orbits
with $t$ as a type.
\begin{definition} \label{d:monotypical}
We denote by $T(\alpha)$ the set of all types that occur in ${A \times [n]}$ (i.e. types $t$ with $N_t >0$).
We call a map $\alpha:\Z^h\lra{} A\wr \Sigma_n$ \emph{monotypical} if $|T(\alpha)|=1$.
\end{definition}
\begin{remark}
When $A = e$ this is equivalent to saying that the $\Z^h$-set of size $n$ classified by $\al$ is isomorphic to a union of isomorphic transitive $\Z^h$-sets.
\end{remark}

\begin{lemma}\label{l:cent_1} There is an isomorphism
$$C_{A\wr \Sigma_{n}}(\im \alpha) \cong \prod \limits_{t\in T(\alpha)} A_t \wr \Sigma_{N_{t}}. $$
\end{lemma}
\begin{proof}
Choose some  $h' \geq 0$  and a surjection $\Z^{h'} \to A$.
Denote by
$$\alpha': \Z^h \oplus \Z^{h'} \to \Sigma_{|A|n}$$
the resulting map.
By Lemma ~\ref{l:wr_1} we have isomorphisms
$$C_{A\wr \Sigma_{n}}(\im \alpha) \cong C_{\Sigma_{n|A|}}(\im \tilde{\alpha}) \cong C_{\Sigma_{n|A|}}(\im \alpha').$$
Also we have a natural bijection
$$T(\alpha) \to  T(\alpha')$$
mapping
$$t \mapsto t'$$
with $A_t \cong A_{t'}$ and $N_{t} = N_{t'}.$
Thus it is enough to consider the case of $A = {0}$.
In this case ${A \times [n]}$ should be considered as a $\Z^{h}$-set
and $C_{\Sigma_n}(\im \alpha)$ is just the group of automorphisms of ${0\times [n]}\cong [n]$ as a $\Z^{h}$-set.
But now $[n]$ decomposes as a disjoint union of $\Z^{h}$-orbits such that there are exactly
$N_t$ of type $t$.
\end{proof}

\begin{lemma}
Let $\alpha: \Z^h \to A \wr \Sigma_n$ be a map, then $\alpha$ in monotypical if and only if the action of $C_{A \wr \Sigma_n}(\im \alpha)$ on ${A \times [n]}$ is transitive.
\end{lemma}
\begin{proof}
Since the action of $C_{A \wr \Sigma_n}(\im \alpha)$ on ${A \times [n]}$ preserves types it cannot be transitive if $\alpha$ is not monotypical. However if $\alpha$ is monotypical of type $t$ it is easy to see that the action of $C_{A \wr \Sigma_n}(\im \alpha) \cong A_t \wr \Sigma_{N_t}$ is transitive. Note that
in this case $$N_t |A_t| = n|A|. $$
\end{proof}

Let $X \coprod Y = \{1,2,...,n\}$ be a non trivial partition and denote by $$\Sigma_{X,Y} \subset \Sigma_n $$
the subgroup of permutations preserving the partition $\{X,Y\}$.

Let $\alpha : \Z^h \to A \wr \Sigma_n$ be a map. Consider the action $\tilde{\alpha}$ of
$\Z^h \oplus A$ on ${A \times [n]}$. Since for every $i \in [n]$,  $A$ acts transitively on $A \times \{i\}$ we have that any $\tilde{\alpha}$-orbit $o$ is of the form $A \times S_o$ for some subset
$S_o \subset [n]$. The sets $S_o$ for all the orbits of $\tilde{\alpha}$ form a partition $$\Coprod{o} S_o = [n].$$
We shall denote this partition by $P_\alpha:=\{S_1,...,S_{k_\alpha}\}$, where $k_{\alpha}$ is the number of $\tilde{\alpha}$ orbits.

\begin{definition}
We say that a map $\alpha : \Z^h \to A \wr \Sigma_n$ is \emph{well-formed}
if $P_{\alpha}= \{S_1,...,S_{k_\alpha}\},$ is such that every $S_i$ is a set of consecutive numbers.
\end{definition}

\begin{lemma}
Every map $\alpha : \Z^h \to A \wr \Sigma_n$ is conjugate to a well-formed map.
\end{lemma}
\begin{proof}
Clear.
\end{proof}

\begin{lemma} \label{nonmono}
Let $\alpha : \Z^h \to A \wr \Sigma_n$ be a map. If $\alpha$ is not monotypical then
there exists some partition $X \coprod Y = \{1,2,...,n\}$ such that the map
$$\alpha: \Z^h \to A \wr \Sigma_n$$ factors through the inclusion
$$A \wr\Sigma_{X,Y} \subset A \wr \Sigma_n $$
and
$$C_{A \wr\Sigma_{X,Y}}(\im \alpha)  = C_{A \wr\Sigma_n}(\im \alpha).$$
\end{lemma}

\begin{proof}
Now let $\alpha: \Z^h \to A \wr \Sigma_n$ be a map. Let $t\in T(\alpha)$ be a type.
%
Now let $t \in T(\alpha)$. Let $B_t \subset {A \times [n]}$ be the set of elements of type $t$.
$B_t$ is a union of $A_t$-orbits. However since we have an inclusion  $A \hookrightarrow A_t$ each $A_t$-orbit is a disjoint union of
$A$-orbits . So there exists some non-empty proper subset $X \subset [n]$ such that
$$B_t = A  \times X.$$
Now the action $C_{A \wr\Sigma_n}(\im \alpha)$ on  ${A \times [n]}$ preserves types. Thus for
$Y := [n] \setminus X$ we see that
$$\alpha: \Z^h \to A \wr \Sigma_n$$ factors through the inclusion
$$A \wr\Sigma_{X,Y} \subset A \wr \Sigma_n $$
and that
$$C_{A \wr\Sigma_{X,Y}}(\im \alpha)  = C_{A \wr\Sigma_n}(\im \alpha).$$
\end{proof}
\begin{cor}\label{c:nonmonotype}
Let $$\alpha : \Z^h \to A \wr \Sigma_n$$ be  a well-formed non-monotypical map then
There exist some $0<m<n$ such that $\alpha$ factors through the inclusion:
$$A \wr \Sigma_m \times \Sigma_{n-m} \subset A \wr \Sigma_n.$$
and we have
$$C_{A \wr \Sigma_m \times \Sigma_{n-m}}(\im \alpha) = C_{ A \wr \Sigma_n}(\im \alpha).$$
\end{cor}

\begin{lemma} \label{monotransfer}
Let $\alpha: \Z^h \to A \wr \Sigma_n$ be a monotypical map of type $$t:\Z^h \oplus A \to A_t.$$
Let $X ,Y$ be a partition of $[n]$.
Then $\alpha$ factors through $A \wr \Sigma_{X,Y}$ if and only if the partition $\{X,Y\}$ is coarser then $P_{\alpha}= \{S_1,...,S_{k_\alpha}\}.$
Furthermore  in this case
$$C_{A \wr \Sigma_{X,Y}}(\im \alpha )=C_{A \wr \Sigma_{X}}(\im \alpha )\times C_{A \wr \Sigma_{Y}}(\im \alpha ). $$
\end{lemma}
\begin{proof}
It is clear that if ${X,Y}$ is not coarser then $P_{\alpha}= \{S_1,...,S_{k_\alpha}\}$ the action of $\alpha$ cannot factor through $A \wr \Sigma_{X,Y}$. Otherwise we get that $A \times [n]$ factors as an $\tilde{\alpha}$ set as
$$A \times [n] = A \times X \coprod A \times Y.$$
The result now follows since

$$C_{A \wr \Sigma_{X,Y}}(\im \alpha ) = C_{\Sigma_{A\times X,A\times Y}}(\im \tilde{\alpha} )  =$$
$$ =
C_{\Sigma_{A\times X}}(\im \tilde{\alpha} )\times C_{\Sigma_{A\times Y}}(\im \tilde{\alpha} ) =
C_{A \wr \Sigma_{X}}(\im \alpha )\times C_{A \wr \Sigma_{Y}}(\im \alpha ). $$
\end{proof}

\begin{cor}\label{c:monotype}
Let $$\alpha : \Z^h \to A \wr \Sigma_n$$ be  a well-formed mono-typical map of type $t$. Let
$l = \frac{|A_t|}{|A|}$ note that we have $N_t = \frac{n}{l}$. Let $0<m<n$, then
\begin{enumerate}
\item If $m$ is divisible by $l$ then  $\alpha$ factors through the inclusion:
$$A \wr \Sigma_m \times \Sigma_{n-m} \subset A \wr \Sigma_n.$$
and we have that the inclusion 
$$C_{A \wr \Sigma_m \times \Sigma_{n-m}}(\im \alpha) \subset C_{ A \wr \Sigma_n}(\im \alpha)$$ is isomorphic to the natural inclusion 
$$A_t \wr \Sigma_{\frac{m}{l}}\times \Sigma_{\frac{n-m}{l}} \subset A_t \wr \Sigma_{\frac{n}{l}}.$$ 
\item  If $m$ is not divisible by $l$ then neither $\alpha$ nor any of its conjugates  factors through the inclusion
$$A \wr \Sigma_m \times \Sigma_{n-m} \subset A \wr \Sigma_n.$$
\end{enumerate}
\end{cor}


\section{The Morava $E$-theory of centralizers}\label{s:ETheoryisFree}
The goal of this section is to prove a freeness result for the Morava $E$-theory of centralizers of tuples of commuting elements in symmetric groups modulo a certain transfer ideal.

All of the groups that we study in this paper are good:
\begin{prop}
Centralizers of tuples of commuting prime-power order elements in symmetric groups are good.
\end{prop}
\begin{proof}
Lemma \ref{l:cent_1} implies that these centralizers are of the form $\Prod{i}A_i\wr\Sigma_{p^{k_i}}$ where $A_i$ is an abelian $p$-group. Now the Sylow $p$-subgroups of this wreath product is $\Prod{i}A_i\wr\Z/p\wr\ldots \wr \Z/p$, which is good. By Proposition 7.2 of \cite{hkr} this implies that the group is good.
\end{proof}

\begin{prop} \label{p.trans}
Let $\al : \Z^h \lra{} \Sigma_{p^k}$ be monotypical (the case $A=e$ in Definition \ref{d:monotypical}) with centralizer $A \wr \Sigma_{p^i}$. Let $I_{tr} \subset \E^0(BA \wr \Sigma_{p^i})$ be generated by maps of the form $A \wr (\Sigma_l \times \Sigma_m) \lra{} A \wr \Sigma_{p^i}$, where $l+m = p^i$ and $l,m>0$. Then $I_{tr} = I_{tr}^{[\al]}$.
\end{prop}
\begin{proof}
This follows immediately from both parts of Corollary \ref{c:monotype}.
\end{proof}

\begin{prop} \label{freeEtheory}
Let $\Z_{p}^{h} \lra{\al} \Sigma_{p^{k}}$. The ring $\E^0(BC(\im \al))/I_{tr}^{[\al]}$ is finitely generated and free as an $\E^0$-module.
\end{prop}
\begin{proof}
Note that Corollary \ref{c:nonmonotype} implies that this statement is trivial for non-monotypical maps $\al$ because $I_{tr}^{[\al]} = \E^0(BC(\im \al))$. When $\al$ is monotypical Proposition \ref{p.trans} implies that we need to study $\E^0(BA \wr \Sigma_{p^i})/I_{tr}$.


This puts us in a situation that is dual to Section 5 of \cite{rezkkoszul} and thus essentially follows from the discussion there. For the purposes of this proof let $E=\E$ so that $E_0 = (\E)_0$. In \cite{rezkcongruence, rezkkoszul} Rezk constructs a functor
\[
\mathbb{T}_m: \text{Mod}_{E_0} \lra{} \text{Mod}_{E_0}
\]
that takes free $E_0$-modules to free $E_0$-modules and with the property that
\[
\mathbb{T}_m(E_0) \cong E_0(B\Sigma_{m}).
\]
More generally on the free $E_0$-module $E_0(BA)$ it takes the value
\[
\mathbb{T}_m(E_0(BA)) \cong E_0(BA \wr \Sigma_{m}).
\]
Now it follows by dualizing Section 5 of \cite{rezkkoszul} that taking the quotient by the image of the transfer maps along
\[
A \wr (\Sigma_i \times \Sigma_j) \lra{} A \wr \Sigma_m
\]
is the \emph{right} linearization of the functor $\mathbb{T}_m$. Let $M$ be a free $E_0$-module and let $i_1,i_2:M \rightarrow M\oplus M$ be the canonical inclusions. The right linearization is defined to be the equalizer
\[
\xymatrix{\mathcal{R}_{\mathbb{T}_m}(M) \ar[r] & \mathbb{T}_m(M) \ar@<.5ex>[rr]^-{\mathbb{T}_m(i_1)+\mathbb{T}_m(i_2)} \ar@<-.5ex>[rr]_-{\mathbb{T}_m(i_1+i_2)} & & \mathbb{T}_m(M\oplus M).}
\]
It follows from Theorem 8.6 in \cite{etheorysym} that
\[
\mathcal{R}_{\mathbb{T}_m}(E_0) \cong (E^0(B\Sigma_m)/I_{tr})^{*}
\]
is free, where $(E^0(B\Sigma_m)/I_{tr})^{*}$ is the $E_0$-linear dual. Since $\mathcal{R}_{\mathbb{T}_m}$ is linear it takes free modules to free modules. Thus
\[
\mathcal{R}_{\mathbb{T}_m}(E_0(BA)) \cong (E^0(BA \wr \Sigma_m)/I_{tr})^{*}
\]
is free. Taking the $E_0$-linear dual gives the desired result.
\end{proof}

\section{Subgroups of $\G_{\E}\oplus \QZ^h$}\label{s:SubGrpisFree}
In this section we prove a freeness result for the ring of functions on the scheme that classifies subgroups of the \pdiv group $\G_{\E} \oplus \QZ^h$.

The main algebro-geometric objects of study are the following:

\begin{definition}
For $k \geq 0$ we define
\[
\Sub_{p^{k}}(\G_{\E}\oplus \QZ^h): \E^0\text{-algebras} \lra{} \text{Set}
\]
to be the functor mapping
\[
R \mapsto \{H  \subset R \otimes (\G_{\E}\oplus \QZ^h), |H| = p^k\}.
\]
\end{definition}

Let
\[
\pr:\G_{\E} \oplus \QZ^h \lra{} \QZ^h
\]
be the projection. This induces a surjective map
\[
\Sub_{p^{k}}(\G_{\E}\oplus \QZ^h) \lra{} \Sub_{\leq p^k}(\QZ^h)
\]
by mapping $H \mapsto \pr(H)$. Since the target of this map is discrete and the map is surjective the fibers disconnect the source. The fibers have the following form:

\begin{definition}
For $A \subset \QZ^h$ of order less than $p^k$, define
\[
\Sub_{p^{k}}^{A}(\G_{\E}\oplus \QZ^h): \E^0\text{-algebras} \lra{} \text{Set}
\]
to be the functor mapping
\[
R \mapsto \{H \subset R \otimes (\G_{\E}\oplus \QZ^h)\}
\]
such that $H$ is a subgroup of order $p^{k}$ and $\pr(H) = A$.
\end{definition}

\begin{example}
When $A = e$ there is an isomorphism $\Sub_{p^k}^{A}(\G_{\E}\oplus \QZ^h) \cong \Sub_{p^k}(\G_{\E})$.
\end{example}

\begin{remark}
Note that $\Sub_{p^k}(\G_{\E}\oplus \QZ^h)$ is a closed subscheme of $\Div_{p^k}(\G_{\E} \oplus \Lk)$.
\end{remark}

Next we show that $\Sub_{p^{k}}^{A}(\G_{\E}\oplus \QZ^h)$ is corepresented by an $\E^0$-algebra that is finitely generated and free as an $\E^0$-module. We rely on Strickland's result:

\begin{thm} \label{thm:strickland}(Theorem 10.1, \cite{subgroups}) The functor $\Sub_{p^k}(\G_{\E})$ is corepresented by an $\E^0$-algebra that is free as an $\E^0$-module of rank equal to the number of subgroups of order $p^k$ in $\QZ^n$.
\end{thm}

\begin{prop} \label{freesubgroups}
The functor $\Sub_{p^{k}}^{A}(\G_{\E}\oplus \QZ^h)$ is corepresentable by an $\E^0$-algebra that is finitely generated and free as an $\E^0$-module.
\end{prop}
\begin{proof}
Note that there is a surjection (there is a trivial section)
\[
\Sub_{p^{k}}^{A}(\G_{\E}\oplus \QZ^h) \lra{} \Sub_{p^{k}/|A|}(\G_{\E})
\]
given by sending a projection $H \lra{} A$ to its kernel.  Let $\bar{\G}_{\E}$ be the pullback of $\G_{\E}$ to $\Sub_{p^{k}/|A|}(\G_{\E})$. The formal group $\bar{\G}_{\E}$ carries the universal subgroup $U$ of order $p^k/|A|$. For this proof let $S = \Gamma \Sub_{p^{k}/|A|}(\G_{\E})$.
Over $\Sub_{p^{k}/|A|}(\G_{\E})$ there is an isomorphism
\[
\Sub_{p^{k}}^{A}(\G_{\E}\oplus \QZ^h) \cong \Hom(A, \bar{\G}_{\E}/U).
\]
We will see this by using the functor of points. Let $\Spec(R)$ be an affine scheme over $\Sub_{p^{k}/|A|}(\G_{\E})$. A $\Spec(R)$ point of the left hand side is a subgroup
\[
H \subset R \otimes_{\E^0} (\G_{\E}\oplus \QZ^h) = R \otimes_S (\bar{\G}_{\E}\oplus \QZ^h)
\]
of order $p^k$ that projects onto $A$. Let $K$ be the kernel of the projection, then $K \cong R \otimes U$. Now we have a map of short exact sequences
\[
\xymatrix{K \ar[r] \ar[d] & H \ar[r] \ar[d] & A \ar@{-->}[d] \\ K \ar[r] & R\otimes (\G_{\E}\oplus \QZ^h) \ar[r] & (R\otimes \G_{\E})/K \oplus \QZ^h.}
\]
Thus we get a $\Spec(R)$ point of the right hand side.

A $\Spec(R)$ point of the right hand side is a map
\[
A \lra{} R \otimes_S \bar{\G}_{\E}/U \cong (R\otimes_{\E^0} \G_{\E})/K.
\]
Combined with the canonical inclusion $A \subset \QZ^h$ this gives a map
\[
A \lra{} (R\otimes \G_{\E})/K \oplus \QZ^h.
\]
Now we pull back to get $H$
\[
\xymatrix{H \ar[r] \ar[d] & R\otimes \G_{\E} \oplus \QZ^h \ar[d] \\ A \ar[r] & (R\otimes \G_{\E})/K \oplus \QZ^h}
\]
a subgroup of order $p^k$ that projects onto $A$. This produces the map back.
Now the scheme $\Hom(A, \bar{\G}_{\E}/U)$ is finite flat over $\Sub_{p^{k}/|A|}(\G_{\E})$ since $\G_{\E}/U$ is $p$-divisible. Theorem \ref{thm:strickland} implies that $\Sub_{p^{k}/|A|}(\G_{\E})$ is finite flat over $\E^0$. Now the composite is finite flat and $\E^0$ is complete local, so we are done.

\end{proof}

\section{A generalized Strickland's theorem} \label{s:GenStrickland}
We show that $\E^0(\Loops^hB\Sigma_{p^k})/I_{tr}$ corepresents the scheme $\Sub_{p^k}(\G_{\E}\oplus \QZ^h)$. In the first subsection we construct the map between the two objects. The map is clearly an isomorphism after inverting $p$. In the next subsection we give a direct proof of the main theorem at height $1$. In the final subsection we prove that the map is an isomorphism by reduction to the height $1$ case.

\subsection{Character theoretic construction of the map}
The purpose of this subsection is to construct a map
\[
f_{p^k}:\Gamma \Sub_{p^{k}}(\G_{\E}\oplus \QZ^h) \lra{} \E^0(\Loops^hB\Sigma_{p^{k}})/I_{tr}.
\]
When $h=0$ this should be compared to Proposition 9.1 of \cite{etheorysym}. We will construct the map by using HKR character theory to embed both sides into class functions. There we will see that the image of the left hand side is contained in the image of the right hand side. Thus there is an injective map between the two rings.

Let
\[
\pi: \Sigma_{p^k} \lra{} U(p^k)
\]
be the permutation representation.

\begin{prop}\label{themap}
The following diagram commutes:
\[
\xymatrix{\Gamma \Div_{p^k}(\G_{\E}\oplus \Lk) \ar@{->>}[r] \ar[dd]^{\cong} & \Gamma \Sub_{p^k}(\G_{\E}\oplus \QZ^h) \ar@{^{(}->}[r] & \Gamma C_0 \otimes \Sub_{p^k}(\G_{\E}\oplus \QZ^h) \ar[d]^{\cong} \\
 & & \Gamma \Sub_{p^k}(\QZ^{n+h}) \ar[d]^{\cong} \\
\E^0(\Loops^hBU(p^k)) \ar[r]^-{\pi^{\ast}} & \E^0(\Loops^hB\Sigma_{p^k})/I_{tr} \ar@{^{(}->}[r] & Cl((\Sigma_{p^k})_{p}^{n+h},C_0)/I_{tr}.}
\]
This implies that there is an injection
\[
f_{p^k}:\Gamma \Sub_{p^k}(\G_{\E}\oplus \QZ^h) \hookrightarrow \E^0(\Loops^hB\Sigma_{p^k}).
\]
\end{prop}
\begin{proof}
We have already seen that the left vertical map is an isomorphism. The right vertical map is given by sending $\al:\Z_{p}^{n+h} \lra{} \Sigma_{p^k}$ to the image of the Pontryagin dual
\[
\QZ^{n+h} \hookleftarrow (\im \al)^*.
\]
The right horizontal maps are injections by Propositions \ref{freeEtheory} and \ref{freesubgroups}. Now the right vertical isomorphism implies that both of the $\E^0$-modules in the middle have the same rank. The left horizontal surjection is due to the fact that we are taking the ring of functions of a closed inclusion.

We can use the results of Section \ref{unitary} to reduce commutativity of this diagram to a height $0$ (and thus combinatorial) problem. Let $\Lk = (\Z/p^k)^h$ and $\Delta_k = (\Z/p^k)^{n+h}$. Consider the following diagram:
\[
\begin{tikzcd}[row sep=large, column sep=0ex]
& \Gamma \Div_{p^k}(\G_{\E} \oplus \Lk) \arrow{dl}\arrow{rr}\arrow{dd} & & \Gamma \Sub_{p^k}(\G_{\E}\oplus \QZ^h) \arrow[dashed]{dl} \arrow{dd} \\
\E^0(\Loops_{k}^{h} BU(p^k)) \arrow[crossing over]{rr}\arrow{dd} & & \E^0(\Loops^hB\Sigma_{p^k})/I_{tr} \\
& C_{0}^0 \otimes \Gamma \Div_{p^k}(\G_{p^{-1}\E}\oplus \Delta_k) \arrow{dl}\arrow{rr} & & \Gamma \Sub_{p^k}(\QZ^{n+h}) \arrow{dl} \\
C_{0}^{0}(\Loops_{k}^{n+h}BU(p^k)) \arrow{rr} & & Cl((\Sigma_{p^k})_{p}^{n+h},C_{0}^0)/I_{tr} \arrow[crossing over, leftarrow]{uu}\\
\end{tikzcd}
\]
The top square of this cube is at height $n$ and the bottom square at height $0$. The left side commutes by Proposition \ref{unitarytheorem}. The front side commutes by Proposition \ref{naturality}. The back side is completely algebro-geometric. The bottom is at height $0$ and thus combinatorial, its commutativity follows from the definition of the permutation representation.

We show that $\Spec(-)$ of the bottom square commutes. Let $\Z_{p}^{n+h} \lra{\al} \Sigma_{p^k}$ be a map classifying a transitive $\Z_{p}^{n+h}$-set. Let $A = \im \al$. Going around the square through $\Sub_{p^k}(\QZ^{n+h})$ sends this to the image of the Pontryagin dual
\[
A^* \lra{} \Lk \subset \QZ^{n+h}
\]
viewed as a divisor in $C_{0}^0 \otimes \G_{p^{-1}\E} \oplus \Lk$ that is just $0$ in $C_{0}^0 \otimes \G_{p^{-1}\E}$.

We show this is the same as going around the bottom square the other way. The map $\al$ factors into the following composite
\[
\Z_{p}^{n+h} \lra{} (\Z/p^k)^{n+h} \lra{g} A \lra{} \Sigma_{p^k}.
\]
Since $A$ is a transitive abelian subgroup of $\Sigma_{p^k}$ it has order $p^k$ and the composite $\rho:A \rightarrow \Sigma_{p^k} \overset{\pi}{\rightarrow} U(p^k)$ is the regular representation. Since $\rho$ is the regular representation it is the sum of all characters of $A$. This defines $|A|$ maps $A \lra{} S^1$ are precisely the elements of $A^*$. Thus we get the divisor
\[
A^* \subseteq \Lk = ((\Z/p^k)^{n+h})^*.
\]
The last thing that needs to be checked is that all of the power series generators (after choosing a coordinate) of
\[
C_{0}^{0}(\Loops_{k}^{n+h}BU(p^k))
\]
map to $0 \in Cl((\Sigma_{p^k})_{p}^{n+h},C_{0}^0)/I_{tr}$. But this is clear as they are induced by elements of higher cohomological degree and the codomain is concentrated in degree $0$. Thus we get exactly the same divisor of $C_{0}^0 \otimes \G_{p^{-1}\E} \oplus \Lk$ by going around the square both ways.
\end{proof}

\subsection{The height $1$ case}
Given an abelian group $A$ and an integer $n \geq 0$, there is a map
$$N_A:A \wr \Sigma_n \to A$$
induced by the addition map $A^n \to A$.
\begin{lemma} \label{lemma:rep}
The map $N_A$ induces an isomorphism
$$R(A) \lra{\cong} R(A \wr \Sigma_n)/ I_{tr}$$
\end{lemma}
\begin{proof}
This isomorphism appears in Section 7 of \cite{zelevinsky}. Note that given a character $\rho$ of $A$.
The pulled back representation on $A \wr \Sigma_n$ is $\rho^{\otimes n} \wr 1$.
\end{proof}

The conjugacy classes of $A \wr \Sigma_n$ have canonical representatives. It will be useful to know what they are.

Every conjugacy class of $\Sigma_n$ is determined by a partition of $n$. We represent this as a non-increasing sequence of integers  $\lambda_1 \geq \lambda_2 \geq \ldots \geq \lambda_r \geq 1$
such that $\sum \lambda_i = n$. Any such conjugacy class contains a unique element that can be written in cycle notation as
$$\sigma = (1,\ldots,\lambda_1)(\lambda_1+1,\ldots,\lambda_1+\lambda_2)\cdots(n-\lambda_r+1,\ldots,n).$$
We call such a $\sigma$ a \emph{canonical representative}.

%
Now let $\sigma$ be a canonical representative of $\Sigma_n$ with $r$ cycles and  let $(a_1,\ldots a_r)\in A^r$ be a tuple of elements in the abelian group $A$.
We denote by $[a_1,\ldots,a_r]\wr \sigma \in A\wr \Sigma_n$
the element
$$(a_1,\overbrace{0,\ldots,0}^{\lambda_1-1},a_2,\overbrace{0,\ldots,0}^{\lambda_2-1},\ldots,a_r,\overbrace{0,\ldots,0}^{\lambda_r-1})\rtimes \sigma. $$
We call the element of the form $[a_1,\ldots,a_r]\wr \sigma \in A\wr \Sigma_n$ a canonical representative in $A\wr \Sigma_n$. It is well-known and easy to check that the set of canonical representatives in $A\wr \Sigma_n$ is a complete set of representatives for the conjugacy classes of $A\wr \Sigma_n$.

\begin{lemma}\label{l:height_0}
Let $A$ be  an abelian group and  $n > 0$ be an integer.
The image of all the transfer maps
$$Cl(A\wr (\Sigma_k \times \Sigma_{n-k}),\mathbb{C}) \to Cl(A\wr \Sigma_n,\mathbb{C})   $$
together generate the ideal of  functions that vanish on all conjugacy  classes in $A\wr \Sigma_n$ which project to the $n$-cycle in $\Sigma_n$ (i.e the conjugacy classes of the form $[a]\wr (1,\ldots,n)$ for some $a \in A$).
\end{lemma}
\begin{proof}
Clear.
\end{proof}

\begin{lemma}\label{l:sgrp_hom}
There exists a canonical  isomorphism of schemes.
$$c_{m,A}:  \Hom(A^*,\G_m) \lra{\cong}  \Sub_{m}^{A^*}(\G_m \oplus \Q/\Z^h). $$
\end{lemma}
\begin{proof}
Denote $l= \frac{m}{|A|}$ and recall the proof of Proposition \ref{freesubgroups}. We get a natural isomorphism
$$\Hom(A^*,\G_m/\G_m[l]) \lra{\cong} \Sub_{m}^{A^*}(\G_m \oplus \Q/\Z^h).  $$
as schemes over $ \Sub_{l}(\G_m)$.
Now $\G_m$ has a unique subgroup $\G_m[l]$ of order $l$ so
$$\Sub_{l}(\G_m) \cong \Spec \Z $$
and we have a canonical isomorphism  $\G_m/\G_m[l]\cong \G_m$.
\end{proof}

\begin{lemma}
There is a canonical isomorphism of schemes
$$\chi_A: \Spec R(A) \lra{\cong} \Hom(A^*,\G_m)$$
\end{lemma}
\begin{proof}
Consider the scheme $\Spec R(A)$ as a functor of points, given a test ring $T$ a map
$R(A) \to T$ is the same as a map $A^* \to T^{\times} = \G_m(T)$.
\end{proof}

Let $l = m/|A|$, then there is an injective map
\[
g_{m,A}:\Gamma \Sub_{m}^{A^*}(\G_m \oplus \Q/\Z^h) \hookrightarrow R(A\wr\Sigma_{l})/I_{tr}
\]
defined by embedding both sides into class functions just as in the proof of Proposition \ref{themap}.

\begin{prop} \label{p.diagram}
Let $\al:\Z^h \lra{} \Sigma_m$ be monotypical and let $A= \im \al$ and $l = m/|A|$. There is a commutative diagram
\[
\xymatrix{\Spec R(A)  \ar[r]_-{\cong}^-{\chi_A} & \Hom(A^*,\G_m)  \ar[d]^-{\cong}_{c_{m,A}}  \\
\Spec(R(A\wr\Sigma_{l})/I_{tr})\ar[u]_-{\cong}^-{N_{A}} \ar[r]^-{g_{m,A}^{*}} & \Sub_{m}^{A^*}(\G_m \oplus \Q/\Z^h) .
}
\]
\end{prop}
\begin{proof}
Since all schemes involved are reduced finite and flat over $\Spec \Z$ it is enough to check the commutativity on $\C$-points. By Lemma \ref{l:height_0} $\Spec(R(A\wr\Sigma_{l})/I_{tr})(\C)$ can be naturally identified with the set of conjugacy classes of $[a]\wr (1,\ldots,l)$ for all $a\in A$.
Let $[a]\wr(1,\ldots,l)$ be such an element. First we would like to describe $g_{m,A}^{*} ([a]\wr (1,\ldots,l))$.
The conjugacy class $[a]\wr(1,\ldots,l)$ corresponds to a map
$$\alpha_a : \Z \to A \wr \Sigma_l$$ such that
$\alpha(a)(1) = [a]\wr(1,\ldots,l)$.
By Section \ref{s:Groups} $\alpha_a$ corresponds to  map
$$\tilde{\alpha}_a:\Z \oplus A \to \Sigma_{l|A|}.$$
To understand the kernel of this map note that $A$ embeds into $A \wr\Sigma_{l}$ via the diagonal map
$z:A \to A \wr\Sigma_{l}$ on the other hand $$([a]\wr (1,\ldots,l))^l = z(a).$$
Thus we get that $B:= \im \tilde{\alpha_a}$ is the result of the pushout:
$$
\xymatrix{
 \Z \ar[d]^{\times l}\ar[r]^{1 \mapsto a}& A\ar@{^{(}->}[d] \\
\Z \ar[r]& B.
}$$
Since $A$ comes with the surjection $\alpha :\Z^h \to A$ we can write $B$ also as the pushout of
$$
\xymatrix{
 \Z\oplus  \Z^h\ar[d]^-{\times l \oplus Id}\ar[r]^-{a \oplus \alpha}& A\ar@{^{(}->}[d] \\
\Z \oplus \Z^h  \ar@{->>}[r]& B.
}$$
Dualizing the diagram above we get
$$
\xymatrix{
B^* \ar@{->>}[d] \ar@{^{(}->}[r] &  \Q/\Z\oplus  \Q/\Z^h \ar[d]^-{\times l \oplus Id}\\
A^* \ar@{^{(}->}[r]_-{a^* \oplus \alpha^*} & \Q/\Z\oplus  \Q/\Z^h.
}$$
Considering $B^*$ as a subgroup in  $\Q/\Z\oplus  \Q/\Z^h$ we get
$$B^* = g_{m,A}^*([a]\wr(1,\ldots,l)) \in \Sub_{m}^{A^*}(\G_m \oplus \Q/\Z^h)(\C).$$

Now we would like to compare $B^*$ with $c_{m,A} \circ \chi_A \circ N_{A}([a]\wr(1,\ldots,l))$.
It is clear that $$N_{A}([a]\wr(1,\ldots,l)) = a \in A = \Spec(R(A))(\C).$$
The map $\chi_A(a)$ is $$a^*:A^* \to \Q/\Z$$ mapping
$a^*(\phi) = \phi(a)$.
To apply $c_{m,A}$ we need to take a pullback as in the proof of Proposition \ref{freesubgroups}.
We get
$$
\xymatrix{
B^* \ar@{->>}[d] \ar@{^{(}->}[r] &  \Q/\Z\oplus  \Q/\Z^h \ar[d]^-{\times l \oplus Id}\\
A^* \ar@{^{(}->}[r]_-{a^* \oplus \alpha^*} & \Q/\Z\oplus  \Q/\Z^h
}$$
just as before.

\end{proof}

\subsection{The map is an isomorphism}

Recall that $C_{0}^0$ is a faithfully flat $p^{-1}\E^0$-algebra. Thus after inverting $p$ the map
\[
f_{p^k}: \Gamma \Sub_{p^k}(\G_{\E} \oplus \QZ^h) \lra{} \E^0(\Loops^h B\Sigma_{p^k})/I_{tr}
\]
constructed in Proposition \ref{themap} is an isomorphism. Now it suffices to show that the map is an isomorphism after base change to any ring in which $p$ is not nilpotent and not invertible. We use the ring $\bar C_{1}^{0}$ constructed in Subsection \ref{sub:K}.

\begin{prop}
The isomorphism of Lemma \ref{lemma:rep} induces an isomorphism
\[
K_{p}^{0}(BA\wr\Sigma_{p^k})/I_{tr} \lra{\cong} K_{p}^{0}(BA).
\]
\end{prop}
\begin{proof}
Let $P$ be a Sylow $p$-subgroup of $G$ and $J$ the kernel of the restriction map $RG \lra{} RP$. By Proposition 9.7 in \cite{chernapprox}, $K_{p}^{0}(BG) \cong \Z_p\otimes R(G)/J$. Let us write $R_pG$ for $\Z_p \otimes RG$. Assume that $A$ is a $p$-group, then we have the following diagram
\[
\xymatrix{I_{tr} \ar@{->>}[r] \ar[d] & J_{tr} \ar[d] \\ R_pA\wr\Sigma_{p^k} \ar@{->>}[r] \ar[d] & K_{p}^{0}A\wr \Sigma_{p^k} \ar[d] \\ R_pA\wr \Sigma_{p^k}/I_{tr} \ar@{->>}[r] \ar[d]^-{\cong} & K_{p}^{0}A\wr \Sigma_{p^k}/J_{tr} \\ R_pA \ar[r]^-{\cong} & K_{p}^{0}A, \ar[u]
}
\]
where $I_{tr}$ is the image of the transfer in representation theory and $J_{tr}$ is the image of the transfer in $p$-adic $K$-theory. The top two horizontal arrows are surjective by Strickland's result. This implies the third horizontal arrow is surjective. But this is a map between free modules of the same rank so it is an isomorphism. This implies that the bottom right arrow is an isomorphism.
\end{proof}

Thus we have a commutative diagram
\[
\xymatrix{K_{p}^{0}(\Loops^h BU(p^k))  \ar@{->>}[d] & \Gamma \Div_{p^k}(\hat{\G}_m \oplus \Lk) \ar@{->>}[d] \ar[l]_-{\cong}\\
K_{p}^{0}(\Loops^hB\Sigma_{p^k})/I_{tr}  & \Gamma \Sub_{p^k}(\hat{\G}_m \oplus \QZ^h), \ar[l]^-{g_{p^k}}_-{\cong}
}
\]
where $\hat \G_m$ is the formal multiplicative group. The bottom arrow may be defined purely in terms of the other arrows. We choose an element in $\Gamma \Sub_{p^k}(\hat{\G}_m \oplus \QZ^h)$, lift it to the global sections of divisors, pass to $K_{p}^{0}(\Loops_{k}^{h} BU(p^k))$ and map down to to $K_{p}^{0}(\Loops^hB\Sigma_{p^k})/I_{tr}$. We slightly abuse notation and call the bottom isomorphism $g_{p^k}$. By Proposition \ref{p.diagram} it is a product of isomorphisms of the form $g_{p^k,A}$.

\begin{thm} \label{t:main}
The map
\[
f_{p^k}: \Gamma \Sub_{p^k}(\G_{\E} \oplus \QZ^h) \lra{} \E^0(\Loops^h B\Sigma_{p^k})/I_{tr}
\]
is an isomorphism.
\end{thm}
\begin{proof}
Recall from Subsection \ref{sub:K} that $\bar C_{1}^0 = \pi_0(L_{K(1)}(C_1 \wedge K_p))$ and that there is an isomorphism of \pdiv groups
\[
\bar C_{1}^0 \otimes \G_{C_1} \cong \bar C_{1}^{0} \otimes \G_m \cong \G_{\bar C_1}.
\]
Applying the character maps of Theorem \ref{charactermap} and Theorem \ref{unitarytheorem} to the square
\[
\xymatrix{\E^0(\Loops_{k}^hBU(p^k)) \ar[d] & \Gamma \Div_{p^k}(\G_{\E} \oplus \Lk) \ar[d] \ar[l]_-{\cong} \\ \E^0(\Loops^hB\Sigma_{p^k})/I_{tr} & \Gamma \Sub_{p^k}(\G_{\E} \oplus \QZ^h) \ar[l]
}
\]
and then base changing to $\bar C_{1}^0$ gives the square
\[
\xymatrix{\bar{C}_{1}^{0}\otimes_{\LE^0} \LE^0(\Loops_{k}^{n+h-1} BU(p^k))  \ar[d] & \bar C_{1}^0 \otimes \Gamma \Div_{p^k}(\G_{\LE} \oplus (\Z/p^k)^{n+h-1}) \ar[l]_-{\cong} \ar@{->>}[d] \\
\bar{C}_{1}^{0}(\Loops^{n+h-1}B\Sigma_{p^k})/I_{tr} & \Gamma \Sub_{p^k}(\G_{\bar C_1} \oplus \QZ^{n+h-1}). \ar[l]_-{\bar{C}_{1}^{0} \otimes f_{p^k}}
}
\]
Proposition \ref{p.basechange} implies that the bottom arrow of this square is the base change $\bar{C}_{1}^{0} \otimes f_{p^k}$ because all of the groups that appear in $\Loops^{n+h-1}B\Sigma_{p^k}$ are good. The map from the top arrow to the bottom arrow factors through
\[
\xymatrix{\bar{C}_{1}^{0}(\Loops_{k}^{n+h-1} BU(p^k))  \ar[d] & \Gamma \Div_{p^k}(\G_{\bar C_1} \oplus (\Z/p^k)^{n+h-1}) \ar[l]_-{\cong} \ar@{->>}[d] \\
\bar{C}_{1}^{0}(\Loops^{n+h-1}B\Sigma_{p^k})/I_{tr} & \Gamma \Sub_{p^k}(\G_{\bar C_1} \oplus \QZ^{n+h-1}). \ar[l]_-{\bar{C}_{1}^{0} \otimes f_{p^k}}
}
\]
From $p$-adic $K$-theory we have the square
\[
\xymatrix{\bar{C}_{1}^{0} \otimes K_{p}^{0}(\Loops_{k}^{n+h-1} BU(p^k))  \ar@{->>}[d] & \bar{C}_{1}^{0} \otimes \Gamma \Div_{p^k}(\hat{\G}_m \oplus (\Z/p^k)^{n+h-1}) \ar[l]_-{\cong} \ar@{->>}[d] \\
\bar{C}_{1}^{0} \otimes K_{p}^{0}(\Loops^{n+h-1}B\Sigma_{p^k})/I_{tr}  & \bar{C}_{1}^{0} \otimes \Gamma \Sub_{p^k}(\hat{\G}_m \oplus \QZ^{n+h-1}). \ar[l]_-{\cong}
}
\]
The bottom arrow of this square is $\bar{C}_{1}^{0} \otimes g_{p^k}$. The last square maps to the middle square and the bottom square of the resulting cube is
\[
\xymatrix{\bar{C}_{1}^{0}(\Loops^{n+h-1}B\Sigma_{p^k})/I_{tr} & \Gamma \Sub_{p^k}(\bar{C}_{1}^0\otimes\G_{ C_1} \oplus \QZ^h) \ar[l]_-{\bar{C}_{1}^{0} \otimes f_{p^k}} \\
\bar{C}_{1}^{0} \otimes K_{p}^{0}(\Loops^{n+h-1}B\Sigma_{p^k})/I_{tr} \ar[u]^-{\cong} & \bar{C}_{1}^{0} \otimes \Gamma \Sub_{p^k}(\hat{\G}_m \oplus \QZ^{n+h-1}), \ar[l]^-{\bar{C}_{1}^{0} \otimes g_{p^k}}_-{\cong} \ar[u]_-{\cong}}
\]
where the vertical arrows are isomorphisms by Proposition \ref{p.basechange}. This square commutes because all of the other squares in the cube commute and the maps $\bar{C}_{1}^{0} \otimes g_{p^k}$ and $\bar C_{1}^0 \otimes f_{p^k}$ are determined by the others.

Thus $\bar C_{1}^0 \otimes f_{p^k}$ is an isomorphism and this implies that $f_{p^k}$ is an isomorphism.

\end{proof}

An advantage of the definition of the map $f_{p^k}$ via the character maps is that the following corollary is immediate.

\begin{cor}
Let $\al: \Z_{p}^h \lra{} \Sigma_{p^k}$ be monotypical, let $A = \im \al$, and let $p^j = p^k/|A|$. Restricting $f$ gives the isomorphism
\[
\E^0(BA\wr \Sigma_{p^j})/I_{tr} \cong \Gamma \Sub_{p^k}^{A^*}(\G_{\E} \oplus \QZ^h).
\]
\end{cor}

\section{Appendix: An elementary proof when $k=1$} \label{s:KeqOne}
We give a proof of Strickland's theorem when $k=1$ by direct calculation. The calculation reduces to a congruence regarding Stirling numbers of the first kind. The trick in this case is that the Honda formal group law is easy to describe modulo $x^{p^n}$.

The mod $I_n$ reduction of the map
\[
f_p:\Gamma \Sub_p(\G_{\E}) \lra{} \E^0(B\Sigma_p)/I_{tr}
\]
from Proposition \ref{themap} is the map
\[
\Gamma \Sub_p(\G_{K(n)}) \lra{} K(n)^0(B\Sigma_p)/I_{tr}.
\]
The domain is still a closed subscheme of divisors. It suffices to show that
\[
K(n)^0(BU(p)) \lra{\pi^*} K(n)^0(B\Sigma_p)/I_{tr}
\]
is surjective.

\begin{prop}
The ring $K(n)^0(B\Sigma_p)/I_{tr}$ is generated by the Chern classes of the permutation representation.
\end{prop}
\begin{proof}
The ideal $I_{tr}$ has rank $1$ so the rank of $K(n)^0(B\Sigma_p)/I_{tr}$ is $(p^n-1)/(p-1)$.

The composite
\[
\Z/p \hookrightarrow \Sigma_p \lra{\pi} U(p)
\]
is the regular representation $\rho$ of $\Z/p$. Thus it suffices to show that the Chern classes of $\rho$ generate a subring of rank $(p^n-1)/(p-1)$ inside of $K(n)^0(B\Z/p)/I_{tr}$.

Recall that there is an isomorphism
\[
K(n)^0(B\Z/p) \cong K(n)^0[x]/(x^{p^n})
\]
and the transfer map from $e$ to $\Z/p$ sends $1$ to $x^{p^n-1}$. Thus
\[
K(n)^0(B\Z/p)/I_{tr} \cong K(n)^0[x]/(x^{p^n-1}).
\]

Let $F$ be the height $n$ Honda formal group law. By Lemma 4.12 in \cite{Bakuradze-Priddy}, there is a congruence
\begin{equation} \label{gplaw}
x+_{F}y = x+y- \sum_{0<j<p}p^{-1} \binom{p}{j}(x^{p^{n-1}})^j(y^{p^{n-1}})^{p-j}
\end{equation}
modulo $x^{p^n}$.

Because the regular representation is the sum of the irreducible representations of $\Z/p$, the total Chern class of $\rho$ is
\[
c(\rho) = \Prod{0<i<p} (1-[i]_{F}(x)t),
\]
where $[i]_F(x)$ is the first Chern class of the tensor power of a generating representation $x:\Z/p \hookrightarrow S^1$.

Now Equation \ref{gplaw} implies that $[i]_F(x) = ix$ mod $x^{p^n}$. We are left trying to understand
\[
c(\rho) = \Prod{0<i<p} (1-ixt),
\]
Thus
\[
c_i(\rho) = \Big(\sum_{0 < k_1 < \ldots < k_i < p} k_1 k_2 \cdots k_i \Big) x^{i} = s(p,i)x^i,
\]
where $s(p,i)$ is the Stirling number of the first kind. It is well-known (see Corollary 4 in \cite{levine}) that $s(p,i)$ is divisible by $p$ when $1 < i < p$.

Thus the only Chern class that does not disappear (we are working in characteristic $p$) is $c_{p-1}(\rho)$. This is congruent to $p-1$.

Now the subring of
\[
K(n)^0(B\Z/p)/I_{tr} \cong K(n)^0[x]/x^{p^n-1}
\]
generated by $x^{p-1}$ has rank $(p^n-1)/(p-1)$. We conclude that $K(n)^0(B\Sigma_p)/I_{tr}$ is generated by the Chern classes of the permutation representation.
\end{proof}

To generalize this, one might want to use the injection of Proposition 9.1 in \cite{etheorysym}
\[
\Gamma \Sub_{p^k}(\G_{\E}) \hookrightarrow \Prod{A \subset \Sigma_{p^k} \text{ transitive}}\Gamma \Level(A,\G_{\E}),
\]
where the product is over abelian transitive subgroups of $\Sigma_{p^k}$.

One of the key obstructions to generalizing this proof seems to be the fact that, even for $p=2$ and $k = 2$, the injection
\[
\Gamma \Sub_{p^2}(\G_{\E}) \lra{} \Gamma \Level(\Z/2\times \Z/2, \G_{\E}) \times \Gamma \Level(\Z/4, \G_{\E})
\]
does not pass to an injection after taking the quotient by $I_n \subset \E^0$.

\bibliographystyle{amsalpha}
\bibliography{mybib}

\end{document}